\documentclass{amsart}
\usepackage{amsmath}
\usepackage{amssymb}
\usepackage{wasysym}
\usepackage{psfrag}
\usepackage{graphicx,epsf,amsmath}  % standard LaTeX graphics tool
\usepackage{epsf,graphicx}
\usepackage{comment}

\def\dd{\displaystyle}
\def\eps{\varepsilon}
\def\pa{\partial}

\newtheorem{thm}{Theorem}[section]
\newtheorem{theorem}{Theorem}[section]
\newtheorem{proposition}[thm]{Proposition}

\newtheorem{lemma}[thm]{Lemma}

\newtheorem{preremark}[theorem]{Remark}
\newenvironment{remark}{\begin{preremark}\rm}{\end{preremark}}
\def\RR{{\mathbb R}}

\def\vphi{\varphi}
\def\vp {v}
\def\na{\nabla}

\title[]{Traveling waves for a boundary reaction-diffusion equation}
\author{L. Caffarelli}
\address{Department of Mathematics, University of Texas at Austin, Austin  TX, 78712, USA}
\email{caffarel@math.utexas.edu}
\thanks{L.C. partially supported by NSF Grant}
\author{A. Mellet}
\address{Department of Mathematics, University of Maryland,
College Park, MD 20742, USA.}
\email{mellet@math.umd.edu}
\thanks{A.M. partially supported by NSF Grant DMS-0901340}
\author{Y. Sire}
\address{Universit\'e Paul C\'ezanne, LATP, Marseille, France }
\email{sire@cmi.univ-mrs.fr}

\begin{document}

\begin{abstract}
We prove the existence of a traveling wave solution for a boundary reaction diffusion equation when the reaction term is  the combustion nonlinearity with ignition temperature.
A key role in the proof is plaid by an explicit formula for traveling wave solutions of a free boundary problem obtained as singular limit for the reaction-diffusion equation (the so-called high energy activation energy limit). 
This explicit formula, which is interesting in itself, also allows us to get an estimate on the decay at infinity of the traveling wave (which turns out to be faster than the usual exponential decay).
\end{abstract}

\maketitle

\section{Introduction}

\subsection{Setting of the problem}
In this paper, we consider the following boundary reaction  equation in the upper half plan $\RR^2_+=\{(x,y)\, ;\, x\in \RR,\; y>0\}$:
\begin{equation}\label{eq:model}
\left\{ 
\begin{array}{ll}
\dd \partial_t u -\Delta u=0,\qquad & \mbox{ in $\mathbb{R}^2_+ \times [0,T]$}\\[5pt]
\dd \frac{\partial u}{\partial \nu}=-f(u),\qquad & \mbox{on $\partial \mathbb{R}^2_+ \times [0,T]$}.
\end{array}
\right.
\end{equation}

We note that, in \eqref{eq:model}, the diffusion takes place in the upper half plane $\RR^2_+$ while the reaction is concentrated along the boundary $y=0$. 
This can be used to model, for instance, the combustion of an oil slick on the ground (or a forest fire), with the temperature diffusing nicely above ground.

The nonlinear reaction term $f$ is a combustion type nonlinearity, satisfying
$$\left\{
\begin{array}{l}
f:[0,1]\rightarrow [0,\infty) \mbox{ is lipschitz  continuous,} \\[5pt]
f(u) >0 \mbox{ for } u\in (0,\alpha) , \quad \mbox{ $f(u)=0$ otherwise,} 
\end{array}
\right.$$
for some $\alpha\in(0,1)$.
We denote
$$M=\int_0^1 f(u)\, du.$$

Note  that the ignition temperature problem usually found in the  literature is written for the temperature $T$ which is related to our unknown $u$ by the relation $T=1-u$. We choose to work with $u$ rather than with $T$ here, because we will be interested in the so-called high energy activation limit. This singular limit is somewhat easier to work with in our setting: It corresponds to the limit $\alpha \rightarrow 0$ with the total mass of $f$ being constant (in other words, $f$ approaches a Dirac mass concentrated at $0$).
This can be achieved by replacing $f$ with
$$ f_\delta(u) = \frac{1}{\delta} f\left(\frac{u}{\delta}\right).$$

In the stationary case, we are then led to 
\begin{equation}\label{eq:stat}
\left\{ 
\begin{array}{ll}
\dd -\Delta u^\delta=0,\qquad & \mbox{ in $\mathbb{R}^2_+ $}\\[5pt]
\dd \frac{\partial u^\delta }{\partial \nu}=-f_\delta(u^\delta),\qquad & \mbox{on $\partial \mathbb{R}^2_+$}, 
\end{array}
\right.
\end{equation}
which is equivalent to the following fractional reaction diffusion equation:
\begin{equation}\label{eq:half}
(-\Delta)^{1/2}w^\delta=-f_\delta(w^\delta) \quad \mbox{in
  $\mathbb{R}$}
\end{equation}   
where $w^\delta(x)=u^\delta(x,0)$.
The singular limit $\delta\to0$ for this problem has been studied, in particular in \cite{CRS}: The solution of (\ref{eq:stat})
 converges,  when $\delta\to 0$, toward a solution of the following free  boundary problem
\begin{equation}\label{eq:FBstat}
\left \{
\begin{array}{lll}
{\dd -\Delta u=0,} & \mbox{ in $\mathbb{R}^2_+$}\\[5pt]
{\dd \frac{\partial u}{\partial \nu} =0,} & \mbox{ in $\Omega(u)=\left \{y=0 \right \} \cap \left \{u>0 \right \} $}, \\[5pt]
\dd \!\!\!\!\!\! \lim_{\tiny \begin{array}{l} x \rightarrow x_0 \\  x \in \Omega(u) \end{array}} \frac{u(x,0)}{|x-x_0|^{1/2}}=2\sqrt{\frac{2M}{\pi}} ,\quad
& \mbox{for all  $x_0 \in
  \partial\Omega(u)$}
\end{array}
\right . 
\end{equation}
(the fact that the free boundary condition is proportional to $M^{1/2}$ follows from a simple scaling argument. The constant $2\sqrt{2/\pi}$ will be found as part of the computations in Section \ref{sec:statreg}).

Similarly, we thus expect the evolution problem \eqref{eq:model} to lead, in the high energy activation limit, to the free boundary problem
\begin{equation}\label{eq:FB}
\left \{
\begin{array}{lll}
{\dd \pa_t u -\Delta u=0,} & \mbox{ in $\mathbb{R}^2_+$}\\[5pt]
{\dd \frac{\partial u}{\partial \nu} =0,} & \mbox{ in $\Omega(u)$}, \\[5pt]
\dd \!\!\!\!\!\! \lim_{\tiny \begin{array}{l} x \rightarrow x_0 \\  x \in \Omega(u) \end{array}} \frac{u(x,0)}{|x-x_0|^{1/2}}=2\sqrt{\frac{2M}{\pi}},\quad
& \mbox{for all  $x_0 \in
  \partial\Omega(u)$.}
\end{array}
\right . 
\end{equation}
The goal of this paper is not to study the convergence of the solutions of \eqref{eq:model} toward those of \eqref{eq:FB} in this high energy activation limit. Instead, we are interested in proving the existence of traveling wave solutions of \eqref{eq:model}. We will however make use of the fact that traveling wave solutions of the free boundary problem \eqref{eq:FB} can be computed explicitly and can be used to control the solutions of \eqref{eq:model}. This is explained in detail in the next section.

\vspace{20pt}

\subsection{Main results}
The goal of this paper is to prove the existence of a traveling wave solution for (\ref{eq:model}), that is a  solution of the form 
$$ u(t,x,y) = \vp (x-ct,y)$$
with 
\begin{equation}\label{eq:limit}
\begin{array}{ll}
\vp (x,y)\longrightarrow 0 \qquad & \mbox{ as } x\rightarrow -\infty, \mbox{ for all $y\in\RR_+$,}\\
\vp (x,y)\longrightarrow 1 \qquad & \mbox{ as } x\rightarrow +\infty,  \mbox{ for all $y\in\RR_+$}.
\end{array}
\end{equation}
In particular, $c$ and $\vp (x,y)$ must solve:
\begin{equation}\label{eq:TW}
\left \{
\begin{array}{lll}
\Delta \vp +c\, \pa_x \vp =0,\,\,\, & \mbox{ in $\mathbb{R}^2_+$}\\[5pt]
\dd \frac{\partial \vp }{\partial \nu} = -f(\vp ),\,\,\,& \mbox{ on $\pa \RR^2_+$},
\end{array}
\right.
\end{equation}
Our main result is the following:
\begin{theorem}\label{thm:1}
There exists  $c>0$ and a function $\vp (x,y)$,  solution of (\ref{eq:limit})-(\ref{eq:TW}).
Furthermore, 
\begin{enumerate}
\item $x\mapsto \vp (x,y)$ is non-decreasing (for all $y\geq0$),
\item $y\mapsto \vp (x,y)$ is non-decreasing (for all $x\in\RR$) and $\lim_{y\to\infty } v(x,y)=1$,
\item there exists $C$ such that 
\begin{equation} \label{eq:asymp}
|1-\vp (x,0) | \leq \frac{2}{\sqrt \pi}\int_{\sqrt{c x}}^\infty e^{- z^2} \, dz \qquad \mbox{ as } x\to +\infty .
\end{equation}
\end{enumerate}
\end{theorem}

\vspace{10pt}

The usual ignition temperature reaction diffusion equation, which reads (still with $u=1-T$)
\begin{equation}\label{eq:rd}
 \pa_t u -\Delta u = - f(u),
 \end{equation}
has been extensively studied. The existence, uniqueness and stability of traveling waves for \eqref{eq:rd} is well known (see, for instance,  \cite{BL2,BL,BLL,BNqual,BNtrav} for existence and uniqueness results, and \cite{BLR,R} for stability analysis).
The singular perturbation limit (high activation energy limit), when $f$ converges to a Dirac mass concentrated at $u=0$, as well as the  analysis of the resulting free boundary problem is studied, for instance, in \cite{BCN,BNS2}.

The study of boundary reaction diffusion equation is more recent.
In \cite{CS}, X. Cabr\'e and J. Sol\`a-Morales study layer solutions for (\ref{eq:model}). These are stationary solutions which are bounded and monotone increasing with respect to $x$.
Many of the tools introduced in \cite{CS} will prove extremely useful in the present paper.
The existence of traveling wave solution for (\ref{eq:model})  is studied by X. Cabr\'e, N. C\'onsul and J. V. Mand\'e \cite{CCM} when $f$ is a bistable nonlinearity. Their proof relies on an energy method, which could also be used in our framework. 
In the present paper, we take  a  different approach.
The main tool is the construction of explicit solutions for the free boundary problems (\ref{eq:FBstat}) and (\ref{eq:FB}) and their regularization into explicit solutions of (\ref{eq:model}) and (\ref{eq:stat}) for particular nonlinearities (which unlike $f$ will not have compact support).
In particular, we will prove:
\begin{proposition}\label{prop:explicitFB}
\item The function
\begin{equation}\label{eq:uu}
u(x,y)=  \frac{1}{\sqrt 2}\left( (x^2+y^2)^{1/2}+x \right)^{1/2}
\end{equation}
is a solution of (\ref{eq:FBstat}) (with $M=\frac{\pi}{8}$).
\item The function
$$v(t,x,y) = \Phi( u(x-c t,y)),\quad \mbox{ with } \Phi(u) = \int_0^u e^{-\frac{\pi}{4} s^2}\, ds$$
is a traveling wave solution of  (\ref{eq:FB})  with speed $c=\frac{\pi}{4}$ (with $M=\frac{\pi}{8}$). 
\end{proposition}
Note that though the formula for $u(x,y)$ may look somewhat complicated, its trace on $\pa\RR^2_+$ is simply given by
$$  u (x,0) = \sqrt{x_+}.$$
These solutions, which are interesting for themselves, provide in particular the  decay estimates (\ref{eq:asymp}). 
We point out that  in the quarter-plane  $\{(x,y)\,;\, x>0\,,\; y>0\}$ the function $w=1-v$ solves:
$$ 
\left \{
\begin{array}{lll}
\Delta w +c\, \pa_x w =0, & \mbox{ in $\{(x,y)\,;\, x>0\,,\; y>0\} $}\\[5pt]
\dd \frac{\partial w }{\partial \nu} =0, & \mbox{ on $\{(x,y)\,;\, x>0\,,\; y=0\}$}
\end{array}
\right.
$$
(assuming, without loss of generality, that $v(0,0)=\alpha$).
Since, $\lim_{x\to\infty} w(x,y) =0 $, $\lim_{y\to\infty} w(x,y) =0$ and $w\leq 1$ it is easy to show that
$1-v(x,y) \leq e^{-cx} $ for  $x>0$, $y>0$. This exponential decay is also the usual decay for solutions of (\ref{eq:rd}).
However, Inequality (\ref{eq:asymp}) shows that we actually have a faster decay at infinity, with in particular
$$ 1-v(x,0) \leq \frac{1}{\sqrt{c\pi}}\frac{e^{-cx}}{\sqrt{x}} .$$

In fact, in Section \ref{sec:last}, we will obtain an exact equivalent for $w$ as $x\to \infty$:
\begin{proposition}\label{prop:as}
There exists a constant $\mu_0>0$ such that
\begin{equation}\label{eq:asymp_exact}
1-v(x,0) = \mu_0\frac{e^{-c x}  }{\sqrt x}  +\mathcal O\left(\frac{e^{-c x}  }{x^{3/2}} \right) \quad \mbox{ as $x\rightarrow \infty$}.
\end{equation}
\end{proposition}

\vspace{10pt}

\begin{remark}
In the same way that (\ref{eq:stat}) was equivalent to (\ref{eq:half}), we can see that if $u$ is a solution of (\ref{eq:model}), then its trace $w(t,x)=u(t,x,0)$ solves
$$ [\pa_t-\Delta]^{1/2} (w) = -f(w).$$
This half heat equation operator is very different from the fractional diffusion equation
$$\pa_t w + (-\Delta)^{1/2} w = -f(w),$$
for which the existence of traveling waves is not obvious. Note that the existence of traveling wave solutions for
$$\pa_t w + (-\Delta)^{s} w = -f(w)$$
was recently established  in \cite{MRS2} when $s\in(1/2,1)$. For  this last problem, it is also shown in \cite{MRS2} that $(1-w)$ decays at $+\infty$ like $\frac{1}{|x|^{2s-1}}$ which is much slower than the exponential decay of the regular diffusion problem.
\end{remark}

\vspace{10pt}

\paragraph{\bf Outline of the paper} The paper is organized as follows: 
In Section \ref{sec:expl}, we will derive explicit formulas for global solutions of \eqref{eq:FBstat} and traveling wave solutions of  \eqref{eq:FB} and thus prove Proposition \ref{prop:explicitFB}.
In Section \ref{sec:statreg}, we will show that we can regularize those solutions 
 to get solutions of \eqref{eq:model} with some very specific reaction term $f(u)$ (without ignition temperature).
Finally, Section \ref{sec:proof} is devoted to the proof of Theorem \ref{thm:1}. 
The general outline of the proof follows classical arguments developed by Berestycki-Larrouturou-Lions \cite{BL2,BLL} (see also Berestycki-Nirenberg \cite{BNtrav}): truncation of the domain and passage to the limit. In that proof we will rely heavily on the results of Section \ref{sec:expl} and \ref{sec:statreg}.
\vspace{20pt}

\section{Explicit solutions of the free boundary problems \eqref{eq:FBstat} and \eqref{eq:FB}: Proof of Proposition \ref{prop:explicitFB}}\label{sec:expl}

\subsection{Explicit stationary solutions of \eqref{eq:FBstat}}

The first part of Proposition \ref{prop:explicitFB} follows from the following lemma:
\begin{lemma}\label{lem:stat}
Let $u(x,y)=Re ((x+iy)^{1/2})=Re(z^{1/2})$. Then $ u$ solves the stationary  free boundary problem \eqref{eq:FBstat}.
\end{lemma} 
\begin{proof}
Since $f(z)=z^{1/2}$ is holomorphic in $\mathbb{C} \backslash \left
  \{ (x,0),x\leq 0 \right \}$, its real part, $ u(x,y)$ is harmonic
 in $\RR^2_+.$ In polar coordinates, we can also write
\begin{equation}\label{eq:fun}
{u}(\rho,\theta)=\rho^{1/2} \cos (\theta /2) =\rho^{1/2} \frac{1}{\sqrt 2} (1+\cos \theta)^{1/2}  ,
\end{equation}
which leads to:
$${u}(x,y)=\frac{(x^2+y^2)^{1/4}}{\sqrt{2}}\left(1+\frac{x}{(x^2+y^2)^{1/2}}\right)^{1/2} = \frac{1}{\sqrt 2}\left( (x^2+y^2)^{1/2}+x \right)^{1/2},$$
and, in particular
$${u}(x,0)=\sqrt{x_+}. $$
 
The function $ u$ thus satisfies the free boundary conditions at the only free boundary point $x_0=0$.
 
It only remains to check that $ u$ satisfies $ \frac{\partial u}{\partial \nu}=-u_y =0$ in $\Omega(u)=\{ u>0\}\cap\{y=0\}$. 
We clearly have $\Omega( u) = \{x>0\}$, and 
using the fact that $ u_x = Re(f'(z))=Re(\frac{1}{2} z^{-1/2})$ and $ u_y =- Im(f'(z))$, we can write:
\begin{equation}
\label{eq:der}
{u}_x= \frac{1}{2} \rho^{-1/2} \cos (\theta /2),\; \mbox{ and  } {u}_y= \frac{1}{2} \rho^{-1/2} \sin (\theta /2),
\end{equation}
which implies in particular that
$${u}_y(x,0) = 0 \mbox { for } x>0 $$ 
and completes the proof of Lemma \ref{lem:stat}.
\end{proof}
\vspace{10pt}

\subsection{Explicit traveling wave solutions of \eqref{eq:FB}}
In order to prove the second part of Proposition \ref{prop:explicitFB}, we want to find a traveling wave solution of \eqref{eq:FB}, that is a solution of the form:
$$\tilde u(t,x,y)=\vphi (x-ct,y).$$
We are going to look for the function $\vphi $ in the form 
$$\vphi (x,y)=\Phi({u}(x,y))$$ 
where $ u$ is the function introduced in Lemma \ref{lem:stat} (stationary solution).
Since $\vphi $ has to solve $\Delta \vphi  +c\, \pa_x \vphi =0$ in $\RR^2_+$,  we must have
$$0= \Delta \vphi  +c\, \pa_x\vphi  =  \Phi''({u})|\nabla u|^2+\Phi'({u})\Delta {u}+c\Phi'({u}) {u}_x.$$
We recall that ${u}$ is harmonic in the upper-half plan, and using formulas \eqref{eq:der} for the derivatives of $ u$, we can check that
$$\frac{{u}_x}{|\nabla {u}|^2}= \frac{1}{2} \rho^{-1/2} \cos(\theta/2) \frac{1}{\frac{1}{4} \rho^{-1} }=  2 \rho^{1/2} \cos(\theta/2)=2{u}.$$
We deduce that the function $\Phi$ must solve
$$\Phi''({u})+c2 {u}\Phi'({u})=0,$$ 
with $\Phi(0)=0$ and $\Phi(+\infty)=1$.
This leads to
\begin{eqnarray}
\Phi_c(u)& = &\frac{2\sqrt{c}}{\sqrt{\pi}}\int_0^u e^{-cs^2}\,ds \nonumber\\ 
& = &  \frac{2}{\sqrt{\pi}}\int_0^{\sqrt{c}u} e^{-s^2}\,ds, \nonumber
\end{eqnarray}
or
\begin{equation}
\Phi_c(u)= \Phi(\sqrt c\, u), \qquad \mbox{ with }\quad \Phi(u)=  \frac{2}{\sqrt{\pi}}\int_0^{u} e^{-s^2}\,ds. \label{eq:Phi}
\end{equation}
We thus have the following proposition (which implies Proposition \ref{prop:explicitFB}):
\begin{proposition}\label{prop:FBTW}
Let $\Phi_c(u)$ be defined by \eqref{eq:Phi}. 
Then the function $\vphi_c (x,y)=\Phi_c({u}(x,y))$ solves the  free boundary problem 
\begin{equation}\label{FBTW}
\left \{
\begin{array}{lll}
\Delta \vp +c \vp _x=0,\,\,\, & \mbox{ in $\mathbb{R}^2_+$}\\[5pt]
\dd \frac{\partial \vp }{\partial \nu} =0,\,\,\,& \mbox{ in $\Omega(\vp )$},\\[5pt]
\displaystyle{\!\!\!\!\!\!\lim_{\tiny \begin{array}{l} x \rightarrow x_0,\\  x \in \Omega(\vp )\end{array}}} \frac{\vp (x,0)}{|x-x_0|^{1/2}}=2\sqrt{\frac{c}{\pi}},\quad & \mbox{ $x_0 \in
  \partial\Omega(\vp )$},
\end{array}
\right . 
\end{equation}
and satisfies  (\ref{eq:limit}).
\item In particular, when $c=2M$, the function $u(t,x,y)=\vphi_c (x-ct,y)$ is a traveling wave solution of \eqref{eq:FB}.
\end{proposition}
\begin{proof}
We only need to check what is happening along the boundary $y=0$:
First, we obviously have
$$ \vphi _\nu (x,0)=\Phi_c'({u}) {u}_\nu (x,0)=0 \qquad\mbox{ in }   \left  \{ {u}> 0\right \}=\left  \{ \vphi > 0 \right \}.
$$
Furthermore, the only free boundary point is $x_0=0$ and we clearly have
$$ \lim_{x \rightarrow 0} \frac{\vphi (x,0)}{x^{1/2}} =\Phi_c'(0) \lim_{x \rightarrow 0} \frac{{u}(x,0)}{x^{1/2}} =2\sqrt{\frac{c}{\pi}}.$$
\end{proof}

\section{Regularization of the solutions of the free boundary problem}\label{sec:statreg}
In this section, we show that one can regularize the explicit solutions of the free boundary problems constructed in the previous section in order to get solutions of the  reaction diffusion equation \eqref{eq:model}  (though not necessarily with the nonlinearity $f$ that we want). 
These regularized solutions will play an important role in the proof of Theorem \ref{thm:1}.

We start with the stationary case:
\begin{proposition}\label{prop:statd}
Recall that $ u (x,y) = Re((x+iy)^{1/2})$ and let
$${u}^\delta(x,y)={u}(x,y+\delta^2).$$ Then ${u}^\delta$ solves the boundary reaction-diffusion equation
\begin{equation}\label{eq:statb}
\left\{ 
\begin{array}{ll}
\dd -\Delta u=0,\qquad & \mbox{ in $\mathbb{R}^2_+ $}\\[5pt]
\dd \frac{\partial u}{\partial \nu}=-\beta_\delta(u),\qquad & \mbox{on $\partial \mathbb{R}^2_+$}, 
\end{array}
\right.
\end{equation}
where
\begin{equation}\label{eq:beta}
\beta_\delta(u)= \frac{1}{\delta} \beta \left(\frac{u}{\delta} \right)\, , \qquad  \beta(u) =\frac{ u}{1+4u^4} .
\end{equation}
\end{proposition} 
Note that we have
$$M=\int_0^\infty \beta_\delta(u)\, du=\frac{\pi}{8},$$
which explains the constant $2 \sqrt{ \frac{2}{\pi}} $  arising in \eqref{eq:FBstat}.

Equation \eqref{eq:statb} is  the same as (\ref{eq:stat}) but with a different nonlinearity. 
We note that $\beta_\delta$ does not have a compact support (no ignition temperature), but decreases as $u^{-3}$ for large $u$.

\begin{proof}
The function $u^\delta$ is clearly harmonic in $\RR^2_+$, so we only have to check the condition along $y=0$.
We note that $y=\delta^2$ is equivalent to $\rho=\frac{\delta^2}{\sin (\theta)}$, and so \eqref{eq:fun} and \eqref{eq:der} yield
$$ u^\delta (x,0) = \frac{\delta}{(\sin(\theta))^{1/2}} \cos(\theta/2)\quad \mbox{ when } x=\delta^2 \frac{\cos \theta}{\sin \theta}$$
and 
$$\frac{\partial u^\delta}{\partial \nu}(x,0)=-{u}_y(x,\delta)=-\frac{1}{2\delta}(\sin \theta) ^{1/2} \sin (\theta/2). $$

Using standard trigonometric formulas we can now check that
$$\frac{\partial u^\delta}{\partial \nu}(x,0) = -\beta_\delta (u^\delta (x,0))$$
with $\beta_\delta$ defined by \eqref{eq:beta}.
  
\end{proof}

We now proceed similarly with the traveling wave solution:
\begin{proposition}\label{prop:TWd}
Let $\Phi_c(u)$ be defined by \eqref{eq:Phi}. 
Then the function $\vphi _{\delta,c}(x,y)=\Phi_c(u^\delta (x,y))$ solves
\begin{equation}\label{eq:TWd}
\left \{
\begin{array}{lll}
\Delta \vp +c\,\pa_x \vp =0, & \mbox{ in $\mathbb{R}^2_+$}\\[5pt]
\dd \frac{\partial \vp }{\partial \nu} = -g_{\delta,c}(\vp ), & \mbox{ on $\pa\RR^2_+$}
\end{array}
\right . 
\end{equation}
where $g_{\delta,c}$ is defined on $[0,1]$ by
$$ g_{\delta,c}(\Phi_c(u)) =  \Phi_c'( u)\beta_\delta(u) \qquad \mbox{ for all } u\in [0,\infty).$$
\end{proposition}
Note  that
$$ \int_0^1 g_{\delta,c} (v)\, dv =  \int_0^\infty g_\delta (\Phi(u)) \Phi'(u)\, du = \int_0^\infty \Phi'( u)^2\beta_\delta(u)\, du $$
which converges to $\frac{\pi}{8}$ as $\delta$ goes to $0$ (but is not equal to $\frac{\pi}{8}$ for $\delta>0$).

\begin{proof}
We have
\begin{eqnarray*}
 \frac{\partial \vphi _{\delta,c}}{\partial \nu} & = & \Phi_c'( u_{\delta,c}) \frac{\partial u_{\delta,c}}{\partial \nu}\\
 & = & - \Phi'_c ( u_{\delta,c})\beta_\delta( u_{\delta,c}),
\end{eqnarray*}
on $\partial \mathbb{R}^2_+ $
and therefore (using the definition of $g_{\delta,c}$)
$$\frac{\partial \vphi _{\delta,c}}{\partial \nu}  = -g_{\delta,c}(\Phi_c( u_{\delta,c})) = - g_{\delta,c}( \vphi_{\delta,c} ) \qquad \mbox{on $\partial \mathbb{R}^2_+ $.}$$
\end{proof}

As a consequence, the function
$$ u_{\delta,c}(t,x,y) = \Phi_c(u^\delta(x-ct,y)),$$
is a traveling wave solution of (\ref{eq:model}) but with a nonlinearity $g_{\delta,c}$ instead of $f$.
This solution will prove very useful in the proof of Theorem \ref{thm:1}, thanks to the following lemma:
\begin{lemma}\label{lem:cd}
The followings hold:
\begin{enumerate}
\item For all $\eta>0$ and $A>0$, there exists $K$ such that if $\delta= A/\sqrt c$ and $c\geq K$, then
\begin{equation*}
g_{\delta,c} (u+\eta) \geq f(u) \qquad \mbox{ for all } u\in[0,1].
\end{equation*}
\item  For all $\eta>0$, there exists  $c_0$ such that
$$ g_{\delta,c} (u)\leq \eta \qquad \mbox{ for all } u\in[0,1]$$
if $\sqrt c<c_0 \delta$.
\end{enumerate}
\end{lemma}
\begin{proof}
The first inequality is equivalent to (with $v$ such that $u+\eta=\Phi_c(v)$)
$$ 
g_{\delta,c}(\Phi_c(v))  \geq f(\Phi_c(v)-\eta)  \qquad \mbox{ for all } v\in[\Phi_c^{-1}(\eta),\infty).
$$
Using the definition of $g_{\delta,c}$ and the fact that $\Phi_c(v) =\Phi(\sqrt c\,v)$, this is equivalent to
$$
 \sqrt{c}\, \Phi'( \sqrt c\, v)\frac{1}{\delta} \beta\left(\frac{v}{\delta}\right) \geq f(\Phi(\sqrt c\, v)-\eta)   \qquad \mbox{ for all } v\in[\Phi_c^{-1}(\eta),\infty)
$$
and so (with $w=\sqrt c v$):
$$
c\, \Phi'(  w)\frac{1}{\sqrt c \,\delta}  \beta\left(\frac{w}{\sqrt c \, \delta}\right) \geq f(\Phi( w)-\eta)   \qquad \mbox{ for all } w\in[\Phi^{-1}(\eta),\infty).
$$
We now take $\delta = A/\sqrt c$, and so we only have to show that 
$$
c\, \Phi'(  w) \frac{1}{A}\beta\left(\frac{w}{A}\right) \geq f(\Phi( w)-\eta)   \qquad \mbox{ for all } w\in[\Phi^{-1}(\eta),\infty).
$$
Since $f(u)=0$ for $u\geq \alpha$ this inequality clearly holds for all $w\geq K_0 =\Phi^{-1}(\alpha+\eta)$.
On the compact set $[\Phi^{-1}(\eta),K_0]$, it is now easy to check that this inequality holds for large $c$ since the left hand side is bounded below
(note that since $f$ is Lipschitz, and so $f(u)\leq K u$, we can show that the choice of $c$ is uniform with respect to $\eta$).

\vspace{10pt}

The second inequality is much simpler: It is equivalent to
$$
c\, \Phi'(  w)\frac{1}{\sqrt c \,\delta}  \beta\left(\frac{w}{\sqrt c \, \delta}\right) \leq \eta \qquad \mbox{ for all } w\in[0,\infty],
$$
and  since $\Phi'(  w)\leq 1$, it is enough to show that
$$%\begin{equation}\label{eq:ineq11}
\, \frac{\sqrt c}{ \delta}  \beta\left(\frac{w}{\sqrt c \, \delta}\right) \leq \eta  \qquad \mbox{ for all } w\in[0,\infty].
$$%\end{equation}
We can thus take
$$ \frac{\sqrt c}{ \delta} \leq c_0:= \frac{\eta}{||\beta||_{L^\infty}}.$$

\end{proof}

\section{Proof of Theorem \ref{thm:1}}\label{sec:proof}
We can now construct a traveling wave solution of (\ref{eq:model}) and prove  Theorem~\ref{thm:1}. 
As in \cite{BLL} (see also \cite{BNtrav,MRS2}), the key steps of the proof are the construction of a solution in a 
truncated domain, and the passage to the limit when this domain goes to $\RR^2_+$.
The solutions constructed in the previous section will play a crucial role as barrier for our problem.
As usual, one of the main difficulty will be to make sure that we recover a finite, non trivial speed of propagation  $c$ at the limit.
\vspace{10pt}

First, we introduce our truncated domain: 
For $R>0$, we denote:
\begin{eqnarray*}
Q_R^+ & = & \{(x,y)\, ;\, -R < x < R \mbox{ and } 0<y< R^{1/4} \}\\
\Gamma^0_R  & = & \{(x,0)\, ;\,  -R< x <R\}=  \pa Q_R^+ \cap \{y=0\} \\
\Gamma^+_R & = &  \pa Q_R^+ \cap \{y>0\}. 
\end{eqnarray*}
Note that the fact that we use a rectangle (rather than, say, a ball) is necessary for the use of the sliding method of \cite{BNplane,BNqual} which will give us the monotonicity of the solutions. The choice of scaling in the $y$-direction ($R^{1/4}$) will be justified shortly.

We now want to solve (\ref{eq:TW}) in $Q^+_R$, but in order to do that, we have to prescribe some boundary condition on $\Gamma^+_R$. It seems natural to use $\vphi_c(x,y)=\Phi_c(u(x,y))$, the traveling wave  solution of the free boundary problem \eqref{FBTW} defined in Proposition~\ref{prop:FBTW}.
However, in order to use the sliding method, it is important that the boundary condition be constant equal to $0$ (respectively $1$) on the lateral boundary $x=-R$ (respectively $x=R$).
We thus define the following function on $Q_R^+$:
$$ \bar \Psi_{c,R} (x,y) = \frac{\vphi_c(x,y)-\vphi_c(-R,R^{1/4})}{\vphi_c(R,0)-\vphi_c(-R,R^{1/4})}.$$
Using the fact that $\vphi_c(x,y)$ is increasing with respect to $x$ (for $y$ fixed) and $y$ (for $x$ fixed), it is easy to check that
$$\bar  \Psi_{c,R} (-R,y) \leq \bar  \Psi_{c,R} (-R,R^{1/4})=0  \quad \mbox{ for all $ y\in[0,R^{1/4}]$},$$
and
$$ \bar \Psi_{c,R} (R,y)\geq \bar \Psi_{c,R} (R,0)= 1 \quad \mbox{ for all $ y\in[0,R^{1/4}]$}.$$
Furthermore, one can check (using the formula \eqref{eq:uu}) that
$$ \begin{array}{ll}
\vphi_c(R,0) \longrightarrow 1 & \mbox{ as } R\to\infty\\
\vphi_c(-R,R^{1/4}) \longrightarrow 0 & \mbox{ as } R\to\infty
\end{array}
$$
(this second limit is what motivated the scaling $R^{1/4}$ in the definition of $Q_R^+$),
and thus 
$$ \bar  \Psi_{c,R} (x,y) \longrightarrow \vphi_c(x,y)  \quad \mbox{ as } R\to\infty
$$
uniformly in $Q^+_R$.
We now truncate $\bar\Psi_{c,R} $ by $0$ and $1$, that is we define
$$ \Psi_{c,R}  = \sup(\inf(\bar\Psi_{c,R}  ,1),0).$$
We then have:
\begin{equation}\label{eq:psi01}
  \Psi_{c,R} (-R,y) = 0 \quad \mbox{ and } \quad \Psi_{c,R} (R,y)=1 \quad \mbox{ for all $ y\in[0,R^{1/4}]$}.
  \end{equation}
and
$$ \Psi_{c,R} (x,y) \longrightarrow \vphi_c(x,y)  \quad \mbox{ as } R\to\infty  \mbox{ uniformly in $Q^+_R$.}
$$
More precisely, a simple computation yields
\begin{equation}\label{eq:psilim}
\sup_{Q^+_R}| \Psi_{c,R} (x,y)-  \vphi_c(x,y)|\leq \frac{1}{\vphi_c(R,0)-\vphi_c(-R,R^{1/4})} -1 \longrightarrow 0   \; \mbox{ as } R\to\infty .
\end{equation}
\vspace{15pt}

Now, for a given $R>0$ and $c>0$, we consider the following problem in $Q_R^+$:
\begin{equation}\label{eq:trunc}
\left \{
\begin{array}{lll}
\Delta \vp +c\,  \pa_x \vp =0,\,\,\, & \mbox{ in $Q_R^+ $}\\[5pt]
\dd \frac{\partial \vp }{\partial \nu} = -f(\vp ),\,\,\,& \mbox{ on $\Gamma^0_R$} \\[5pt]
\vp =\Psi_{c,R} (x,y)  &  \mbox{ on $\Gamma^+_R$.}
\end{array}
\right.
\end{equation}
\vspace{10pt}

The proof of Theorem \ref{thm:1} is as follows: We  first prove that for all $R$ there is a unique $c_R$ such that the solution of \eqref{eq:trunc} satisfies $v(0,0)=\alpha$. We then pass to the limit $R\to\infty$ and check that the limit is the solution we were looking for.
\vspace{10pt}

\subsection{Solutions of the truncated problem \eqref{eq:trunc}}
In this section, we prove the following proposition:
\begin{proposition}\label{prop:trunc}
There exists $R_0$ such that for all $R\geq R_0$ there exists  $c_R$ such that the corresponding solution $\vp _R$ of (\ref{eq:trunc}) satisfies
$$0\leq v(x,y)\leq 1\, , \qquad  \vp _R (0,0)=\alpha.$$
Furthermore,
\begin{enumerate}
\item $\vp _R$ is in $\mathcal C^{1,\alpha}(\overline{Q^+_{R/2}})  $ for all $\alpha\in(0,1)$ and 
\begin{equation}\label{eq:estimate}
|| v_R||_{\mathcal C^{1,\alpha}(\overline{ Q^+_{R_0}}) } \leq C(R_0) \qquad \mbox{ for all } R\geq 2R_0.
\end{equation}
\item There exists $K$ such that $0< c_R\leq K$.
\item The function $\vp _R$ is non-decreasing with respect to $x$ (for all $y$).
\end{enumerate}
\end{proposition}

First, we show:
\begin{lemma}\label{lem:ex}
For all $R>0$ and all $c>0$, Equation (\ref{eq:trunc}) has a unique solution $\vp _c(x,y)$.
Furthermore, $0\leq v(x,y)\leq 1$, $\vp _c(x,y)$ is increasing with respect to $x$ (for all $y$) and the function 
$c\mapsto \vp _c(0,0)$ is continuous with respect to $c$.
\end{lemma}
\begin{proof}
It is readily seen that the function  $u= 0$ and $u=1$ are respectively sub and super-solution of (\ref{eq:trunc}).  The existence of a solution can thus be established, for instance, using Perron's method.
Using (\ref{eq:psi01}) and the fact that $\Psi_{c,R} (x,y)$ is monotone increasing with respect to $x$, the classical sliding method of \cite{BNqual,BNplane} shows that the solution is unique and that the function $x\mapsto \vp _c(x,y)$ is increasing (for all $y$). 
\end{proof}

It is clear that $v_R$ is smooth inside $Q_R^+$.  
In order to study the regularity of $\vp _R$ up to the boundary $\Gamma^0_R$ and derive \eqref{eq:estimate}, we use a very nice tool introduced by  Cabr\'e and Sol\`a-Morales in \cite{CS}: 
We note that the function
$$w(x,y) = \int_0^y \vp_R (x,z)\, dz$$
is solution to
$$ 
\left\{
\begin{array}{ll}
\Delta w+ c\pa_x w = f(\vp_R (x,0)) &\mbox{ in $Q_R^+$}\\[5pt]
w = 0 & \mbox{ on $\Gamma^0_R$}
\end{array}
\right.
$$
and so the odd reflexion $\bar w$ of $w$ (with respect to the $x$ axis) solves
$$
\Delta \bar w+ c\pa_x \bar w = f(\vp (x,0)) \mbox{ in $Q_R$.}
$$
Using standard regularity results, it is then not difficult to show:
\begin{lemma}\label{lem:comp}
If $f$ is Lipschitz, then $\bar w$ is in $\mathcal C^{2,\alpha}(Q_{R/2})$ for all $\alpha\in (0,1)$ and so
$\vp _R$ is in $\mathcal C^{1,\alpha}(\overline{Q^+_{R/2}})$.

Moreover, for all $R_0$, there exists $C(R_0)$ such that \eqref{eq:estimate} holds.
\end{lemma}
\begin{proof} The proof is the same as that of  Lemma 2.2 in \cite{CS}. We recall it here:
First, since $ f(\vp (x,0))\in L^\infty$, we have  $\bar w  \in W^{2,p}(Q_{R/2})$ for all $p<\infty$ and so $\bar w  \in \mathcal C^{1,\alpha}(Q_{R/2})$ (Sobolev embeddings).  It follows that $v_R$ is in $C^{\alpha}(Q_{R/2})$ and since $f$ is Lipschitz, we deduce $ f(\vp_R (x,0))\in C^{\alpha}(Q_{R/2})$. 
Classical Shauder estimates now yields $\bar w  \in \mathcal C^{2,\alpha}(Q_{R/2})$ which completes the proof.
\end{proof}
Note that if $f \in\mathcal C^{1,\alpha}$, then we can show  $\vp _R\in \mathcal C^{2,\alpha}(\overline{Q^+_{R/2}})$.

\vspace{10pt}

It remains to show that we can choose $c$ so that $v_R(0,0)=\alpha$. This will follows from the following lemma:
\begin{lemma}\label{lem:speed}
There exists $R_0$ and $K$ such that if $R\geq R_0$, then the solution of (\ref{eq:trunc}) satisfies
\begin{enumerate}
\item If $c\geq K$ then $\vp _c(0,0)> \alpha $.
\item  $\liminf_{c\to 0^+} \vp _c(0,0)< \alpha $.
\end{enumerate}
\end{lemma}
Lemma \ref{lem:speed}, together with the continuity of $\vp _c(0,0)$ with respect to $c$ implies that there exists $c\in(0,K]$ such that $\vp _c(0,0)=\alpha$ and thus completes the proof of Proposition \ref{prop:trunc}

\begin{proof}
In this proof, we use some results of the previous section:
We recall that
$$\vphi _{\delta,c}(x,y)= \Phi_c(u^\delta(x,y))= \Phi(\sqrt{c} \, u^\delta(x,y))
$$
solves
$$
\left \{
\begin{array}{lll}
\Delta \vp +c \, \pa_x \vp =0,\,\,\, & \mbox{ in $\RR^2_+$}\\[5pt]
\dd \frac{\partial \vp }{\partial \nu} = -g_{\delta,c}(\vp ) ,\,\,\,& \mbox{ on $\pa\RR^2_+$}.
\end{array}
\right.
$$
Next, for some fixed $\eta<(1-\alpha)/2$, we define 
$$ w(x,y)= (\vphi _{\delta,c}(x,y) -\eta)_+.$$
It is readily seen that $w$ satisfies
$$ \Delta w +c w_x \geq 0 \mbox{ in } \RR^2_+$$
and 
$$  \frac{\partial w}{\partial \nu} \leq 0 = -f(w) \mbox { in $\pa\RR^2_+ \cap \{w=0\}$}.$$
Furthermore, on $\pa\RR^2_+ \cap \{w>0\}$, we have
$$ \frac{\partial w}{\partial \nu} = -g_{\delta,c}(\vphi _{\delta,c}) =-g_{\delta,c}(w+\eta).$$
Lemma \ref{lem:cd}, thus implies
$$  \frac{\partial w}{\partial \nu} \leq -f(w) \quad \mbox { on $\pa\RR^2_+\cap\{w>0\}$}$$
provided $\delta = A/\sqrt c$ and $c\geq K(A)$ ($A$ will be chosen later).

Finally, on $\Gamma^+_R$, we have 
$$ w= (\vphi_{c,\delta}(x,y) -\eta)_+ \leq \vphi_c(x,y) -\eta/2 \leq \Psi_{c,R} (x,y).$$
The first inequality is satisfied provided $\delta$ is small enough (which can be ensured, possibly by requiring $c>K'>K$), while the second inequality is satisfied for large $R$ (using \eqref{eq:psilim}).

The maximum principle thus yields
$$ \vp _c \geq w \mbox{ in } Q_R^+.$$
Finally, we note that 
$$ w(0,0) = (\vphi _{\delta,c}(0,0) -\eta)_+ = (\Phi(\sqrt{c} \, u^\delta(0,0)) -\eta)_+ = (\Phi(\sqrt{c} \, \delta) -\eta)_+= (\Phi(A) -\eta)_+$$
If thus only remains to choose $A$ large enough  so that $\Phi(A) -\eta>\alpha$ (which is possible since  $\eta<(1-\alpha)/2$).

\vspace{20pt}

Next, for some fixed $\eta<\alpha/2$, we define
$$w(x,y)=\vphi _{\delta,c}(x,y) + \frac{\eta}{R}(R-y).$$
Since $y\leq R/2$ in $Q_R^+$, we have (using \eqref{eq:psilim})
$$ w(x,y) \geq \vphi _{\delta,c}(x,y) +\eta/2\geq \vphi _{c}(x,y) +\eta/2\geq \Psi_{c,R} (x,y)\quad \mbox{ on $\Gamma^+_R$},$$
if $R$ is large enough.

Furthermore, it is readily seen that $w$ satisfies
$$ \Delta w +c \,\pa_x w \geq 0 \mbox{ in } \RR^2_+$$
and 
$$  \frac{\partial w}{\partial \nu} = -g_{\delta,c}(w-\eta ) +\eta/R  \mbox { in $\pa\RR^2_+$}.$$
Lemma \ref{lem:cd}, thus implies
$$  \frac{\partial w}{\partial \nu} \geq 0 \geq -f(w) \quad \mbox { on $\pa\RR^2_+$}$$
provided  $\sqrt{c}\leq c_0(\eta/R) \delta$.

We deduce
$$ \vp _c(0,0 ) \leq w(0,0) = \Phi(\sqrt c\, \delta) +\eta .$$
The result follow easily by taking $\delta$ and $c$ small enough (recall that  $\eta<\alpha/2$ and $\Phi(0)=0$).
\end{proof}

\vspace{10pt}

\subsection{Passage to the limit $R\to \infty$}
In order prove Theorem~\ref{thm:1}, we have to pass to the limit $R\rightarrow\infty$ in the truncated problem.

More precisely,  Theorem \ref{thm:1} will follow from the following proposition:
\begin{proposition}\label{prop:limit}
Under the conditions of Proposition \ref{prop:trunc},
there exists a subsequence $R_n\rightarrow \infty$ such that $\vp _{R_n}\longrightarrow \vp _0$ (uniformly on every compact set) and $c_{R_n} \longrightarrow c_0$. The function $v_0$  solves (\ref{eq:TW}), and 
\begin{enumerate}
\item $c_0 \in (0,K]$
\item $x\mapsto v_0 (x,y)$ is non-decreasing (for all $y\geq0$)
\item $y\mapsto v_0 (x,y)$ is non-decreasing (for all $x\in\RR$) 
\item $v_0$ satisfies
$$
\begin{array}{ll}
\vp_0 (x,y)\longrightarrow 0 \qquad & \mbox{ as } x\rightarrow -\infty \\
\vp_0 (x,y)\longrightarrow 1 \qquad & \mbox{ as } x\rightarrow +\infty .
\end{array}
$$
Furthermore, (\ref{eq:asymp}) holds and $\lim_{y\to +\infty} v_0(x,y)=1$.
\end{enumerate}
\end{proposition}
This section is devoted to the proof of this proposition.
\vspace{10pt}

First, we recall that $ c_R \in (0,K]$, and so Proposition \ref{prop:trunc} (1) implies that there exists a subsequence $R_n\rightarrow \infty$ such that 
$$ 
\begin{array}{l}
c_{n} :=c_{R_n} \longrightarrow c_0\in[0,K]\\[5pt]
 \vp _n := \vp _{R_n}  \longrightarrow \vp _0 \quad \mbox{ uniformly on every compact set}
 \end{array}
 $$
as $n\rightarrow \infty$. 
It is readily seen that $\vp _0 \in \mathcal C^{1,\alpha}(\RR^2_+) $ solves
\begin{equation}\label{eq:nontrunc}
\left \{
\begin{array}{lll}
\Delta \vp +c_0 \pa_x \vp =0,\,\,\, & \mbox{ in $\RR^2_+ $}\\[5pt]
\dd \frac{\partial \vp }{\partial \nu} = -f(\vp ),\,\,\,& \mbox{ on $\pa\RR^2_+$} 
\end{array}
\right.
\end{equation}

Furthermore, we can show:
\begin{lemma}\label{lem:grad}
There exists $C$ such that
$$ |\na v_0 (x,y)| \leq C \qquad \mbox{ for all $(x,y)\in \RR^2_+$}.$$
\end{lemma}
\begin{proof}
Indeed, Proposition \ref{prop:trunc} (1) (with $R_0=2$) gives $|\na v_0|\leq C$ is $Q^+_1(a,0)$ for all $a$ and so $|\na v_0|\leq C$ in $\{0\leq y\leq 1\}$. Interior gradient estimates and the fact that $||v_0||_{L^\infty} \leq 1$ gives the result.
\end{proof}
Note that interior gradient estimates (in $B_t(x,t)$) also yield
$$ |\na u (x,t)| \leq \frac{1}{t},$$
and so $|\na u(x,y)|\to0$ as $y\to \infty$. 

Proposition \ref{prop:trunc} (3) implies that $x\mapsto v_0(x,y)$ is non-decreasing with respect to $x$ (for all $y$), and we can show:
\begin{lemma}
The function $y\mapsto v_0(x,y) $ is non-decreasing with respect to $y$.
\end{lemma}
\begin{proof}
We note that $w=\pa_yv_0$ is solution of 
$$\left \{
\begin{array}{lll}
\Delta w +c_0 \pa_x w =0,\,\,\, & \mbox{ in $\RR^2_+ $}\\[5pt]
\dd w  = f(\vp )\geq 0 ,\,\,\,& \mbox{ on $\pa\RR^2_+$.} 
\end{array}
\right.
$$
Using the fact that $w=\pa_yv_0$ is bounded in $\RR^2_+$ (Lemma \ref{lem:grad}) and continuous up to the boundary (Lemma \ref{lem:comp}),
we deduce that $w=\pa_yv_0\geq 0 $ in $\RR^2_+$, hence the lemma.  
\end{proof}

We now have to show that $c_0>0$ and that $v_0$  has the appropriate limiting behavior as $x\to\pm\infty$.

We start with the following lemma, which is reminiscent of Lemma 3.3 in \cite{CS}:
\begin{lemma}\label{lem:limit}
There exists $\gamma^+$, $\gamma^-$  such that 
\begin{equation}\label{eq:lim0}
\lim_{x\rightarrow \pm\infty}\vp _0(x,0) = \gamma^\pm
\end{equation}
with
\begin{equation}\label{eq:limineq}
0\leq \gamma^-\leq \alpha\leq\gamma^+\leq 1.
\end{equation}
Furthermore, for all $R>0$, 
\begin{equation}\label{eq:limunif}
 ||\vp _0-\gamma^\pm||_{L^\infty(Q_R^+(x,0))} \longrightarrow 0 \qquad \mbox{ as } x\to\pm\infty
 \end{equation}
and 
\begin{equation}\label{eq:limgrad}
 ||\nabla \vp _0||_{L^\infty(Q_R^+(x,0))} \longrightarrow 0 \qquad \mbox{ as } x\to\pm\infty.
 \end{equation}
 Finally, 
 \begin{equation}\label{eq:limf}
 f(\gamma^+)=f(\gamma^-)=0.
 \end{equation}
\end{lemma}
\begin{proof}
The existence of $\gamma^+$ and $\gamma^-$  
follows from the fact that $x\mapsto \vp _0(x,y)$ is monotone increasing and bounded by $0$ and $1$.
The fact that $\vp _0(0,0)=\alpha$ implie~\eqref{eq:limineq}.

Next, we see that \eqref{eq:limunif} and  \eqref{eq:limgrad} are equivalent to 
$$ 
 ||\vp _0^t-\gamma^\pm||_{L^\infty(Q_R^+(0,0))}+ ||\nabla \vp _0^t||_{L^\infty(Q_R^+(0,0))}\longrightarrow 0 \qquad \mbox{ as } t\to\pm\infty$$
 where 
$$ \vp _0^t(x,y) = \vp _0(x+t,y).$$
The proof then follows from a simple compactness argument:
We assume that
$$ 
\liminf_{n\to\infty } ||\vp _0^{t_n}-\gamma^\pm||_{L^\infty(Q_R^+(0,0))}+ ||\nabla \vp _0^{t_n}||_{L^\infty(Q_R^+(0,0))}\geq \eps>0 \qquad \mbox{ as } n\to \infty$$
for some sequence $t_n\to \infty$.
Then Lemma \ref{lem:comp} gives  $\mathcal C^{1,\alpha}$ estimates on $\vp _0^{t_n}$ in $Q_R^+(0,0)$ and therefore a subsequence of $\vp _0^{t_n}$ converges to some function $\vp _1^\pm$ satisfying
$$ 
||\vp _1^\pm-\gamma^\pm||_{L^\infty(Q_R^+(0,0))}+ ||\nabla \vp _1^\pm||_{L^\infty(Q_R^+(0,0))}\geq \eps
$$
However, by \eqref{eq:lim0}, we have 
$ \vp _1^\pm = \gamma^\pm$ on $\Gamma_0$, and since $\vp _1^\pm$ is a bounded solution of  
$$ \Delta v+c_0\pa_x v  = 0 \quad \mbox{ in } \RR^2_+,$$
we deduce that
$\vp _1^\pm=\gamma^{\pm}$ in all of $\RR^2_+$ 
(this is a consequence of Liouville's theorem, after extending the function $\vp _1^\pm-\gamma^{\pm}$ to $\RR^2$ by an odd reflection).
This is a contradiction with the inequality above.

Finally, this argument also implies that $\gamma^\pm$ is a solution of  \eqref{eq:nontrunc} which gives~\eqref{eq:limf}.
\end{proof}

\vspace{10pt}

The crucial step in the proof of Proposition \ref{prop:limit} is now to prove that $c_0>0$. This will be in particular a consequence of the following lemma:
\begin{lemma}\label{lem:intc}
Let $\vp _0$ be a  solution of \eqref{eq:nontrunc}. Then $\pa_x \vp _0 \in L^2(\RR^2_+)$ and the following equality holds:
\begin{equation}\label{eq:intc}
\int_{\gamma_-}^{\gamma_+} f(s)\, ds =  c_0 \int_{\RR^2_+} |\pa_x \vp _0|^2 \, dx\, dy
\end{equation}
\end{lemma}
\begin{proof}
Again, the idea of the proof comes from \cite{CS} (though in that paper, it is assumed that $c_0=0$).
We write $v=v_0$.
Then, multiplying the equation by $v_x$ leads to
$$v_{xx} v_x  + v_{yy} v_x  +c_0 (v_x) ^2 = 0 $$
which can be rewritten as
\begin{equation}\label{eq:cs1}
 \pa_{x} (\frac{1}{2} (v_x)^2) + \pa_{y}(v_y v_x ) - \pa_{x}  (\frac{1}{2} (v_y )^2 )  +c_0 (v_x)^2 = 0 
\end{equation}
Integrating with respect to $y\in(0,R)$, we deduce:
\begin{equation}\label{eq:cs2}
\frac{ d}{dx}\int_0^R \frac{1}{2} [(v_x )^2 -(v_y)^2 ] \, dy -v_y(x,0)v_x (x,0 ) +v_y(x,R)v_x (x,R )   +c_0 \int_0^R (v_x )^2 \, dy = 0 
\end{equation}
and so, defining  $F'=f$, we get
$$
\frac{ d}{dx} \int_0^R \frac{1}{2}  [(v_x )^2 -(v_y)^2 ]   \, dy -\frac{ d}{dx} F( v(x,0 ))   +v_y (x,R)v_x (x,R ) +c_0 \int_0^R (v_x)^2 \, dy = 0 .
$$
which we rewrite as 
\begin{equation}\label{eq:cs3}
 v_y (x,R)v_x (x,R ) +c_0 \int_0^R (v_x)^2 \, dy  = - \frac{ d}{dx} \int_0^R \frac{1}{2}  [(v_x )^2 -(v_y)^2 ]   \, dy + \frac{ d}{dx} F( v(x,0 ))  .
\end{equation}

Using \eqref{eq:limgrad}, we see that 
$$ \int_0^R  \frac{1}{2} [(v_x )^2 -(v_y)^2 ]     \, dy \longrightarrow 0 \quad \mbox{ as } x\rightarrow \pm\infty$$
and so the right hand side in \eqref{eq:cs3} is integrable with respect to $x\in \RR$.
Since the left hand side is non-negative (recall that $u_x\geq 0$ and $u_y\geq 0$), we deduce that it is also integrable, and 
$$ \int_{\RR} v_y (x,R)v_x (x,R )\, dx + c_0 \int_{\RR}\int_0^R  |v_x| ^2 \, dy\, dx =  F(\gamma_+)  -F( \gamma_-) = \int_{\gamma_-}^{\gamma_+} f(s)\, ds .$$
Passing to the limit $R\to\infty$, we deduce that $v_x  \in L^2(\RR^2_+)$  and
 obtain (\ref{eq:intc}).
\end{proof}

We can now show:
\begin{lemma}
The limiting speed $c_0$ satisfies
$$c_0>0$$
\end{lemma}
\begin{proof} 
First, we note that for any given $n$, there exists $x_n$ such that
$$ \vp _n (x_n,0)=\alpha/2.$$
We thus consider the sequence
$$\psi_n(x,y) = \vp _n(x+x_n,y).$$
Proceeding as before, it is readily seen that up to another subsequence, we can assume that $\psi_n$ converges uniformly to a function $\psi_0$ satisfying in particular
$$ \psi_0(0,0) = \alpha/2.$$
The function $\psi_0$ satisfies the same equation as $\vp _0$ but we have to be careful with the domain. We note that up to a subsequence, we can assume that $x_n-R_n$ is either convergent or goes to $+\infty$.
We need to distinguish the two cases:
\medskip

\noindent {\bf Case 1: $x_n-R_n\rightarrow+ \infty$:} In that case, $\psi_0$ solves \eqref{eq:nontrunc} in $\RR^2_+$.
Furthermore, $\psi_0(0)=\frac{\alpha}{2}$ and $\psi_0$ is monotone increasing with respect to $x$. In particular, proceeding as before, it is readily seen that there exists $\bar \gamma_-$ and $\bar \gamma_+$ such that   $\lim_{x\rightarrow \pm \infty}\psi_0(x,0) =\bar \gamma_\pm$ with
$$0\leq \bar \gamma_-\leq\frac{\alpha}{2}\leq\bar \gamma_+\leq 1.$$
Using Lemma \ref{lem:limit}, we get
$$ f(\bar  \gamma_-) = f (\bar  \gamma_+) = 0$$
and so 
$$ \bar \gamma_-=0 \qquad \mbox{ and }\qquad \bar \gamma_+ \geq \alpha.$$
Furthermore, Lemma \ref{lem:intc}, yields
$$
\int_{\bar\gamma_-}^{\bar \gamma_+} f(u)\, du = c_0 \int_{\RR^2_+} |\pa_x \psi_0|^2 \, dx\, dy
$$
and so 
$$ c_0 \int_{\RR^2_+} |\pa_x \psi_0|^2 \, dx\, dy  \geq \int_0^\alpha f(u)\, du>0.$$
which gives the lemma  in this first case.
\medskip

\noindent {\bf Case 2: $x_n-R_n\rightarrow L<\infty$:}  
 In that case, $\psi_0$ solves
\begin{equation}\label{eq:psi02}
\left \{
\begin{array}{lll}
\Delta \vp +c_0 \pa_x \vp =0,\,\,\, & \mbox{ for $x> -L  $, $y>0$}\\[5pt]
\dd \frac{\partial \vp }{\partial \nu} = -f(\vp ),\,\,\,& \mbox{ for $x>-L $ and $y=0$}
\end{array}
\right.
\end{equation}
and we need to modify the proof slightly. 
Proceeding as in the first case, we have that
$$ \lim_{x\rightarrow + \infty}\psi_0(x,0) =\bar \gamma_+ \geq \alpha,$$
and we notice that  
 $\psi_0(-L,y)=0$ for $y>0$.
We then proceed as in the proof of Lemma \ref{lem:intc}: Integrating  \eqref{eq:cs3}  for $x\in [-L,\infty]$  and letting $R\to\infty$ and get
$$c_0 \int_{[-L,\infty]\times(0,\infty)}  (\pa_x \psi_0)^2 \, dx\, dy \geq F( \bar \gamma_+)-F(0)   + \int_0^\infty (\pa_x \psi_0 )^2(-L,y)\, dy$$
and so
$$c_0 \int_{[-R,\infty]\times(0,\infty)}  (\pa_x \psi_0)^2 \, dx\, dy \geq  \int_{0}^{\alpha} f(u)\, du >0
$$
which gives the result in the second case.
\end{proof}

Finally the following lemma concludes the proof of Proposition \ref{prop:limit}
\begin{lemma}\label{lem:limitg}
We have
$$ \gamma^+=1 \mbox{ and } \gamma^-=0.$$
Furthermore,  \eqref{eq:asymp} holds and  $\lim_{y\to +\infty} v_0(x,y)=1$.
\end{lemma}
\begin{proof}
Using (\ref{eq:psilim}), we see that there exists a constant $\eta(R_n)$ such that
$$w(x,y):= \vphi_{c_n}(x,y) -\eta(R_n)\leq  \Psi_{c_n,R_n} (x,y) \qquad \mbox{ in $Q_{R_n}^+$}$$
and 
$$ \eta (R_n)\longrightarrow 0 \quad \mbox{ as } R_n\to \infty.$$
Furthermore, it is readily seen that $w$ solves
$$
\left \{
\begin{array}{lll}
\Delta w+c_n \pa_x w=0,& \mbox{ in $Q_{R_n}^+$}\\[5pt]
w \leq \Psi_{c_n,R_n}  = v_n & \mbox{ on $\Gamma^+_{R_n}$}\\[5pt]
\dd \frac{\partial w}{\partial \nu} = 0 ,& \mbox{ for $y=0$ and $x>0$} \\[5pt]
\dd w\leq 0 \leq v_n & \mbox{ for $y=0$ and $x<0$}
\end{array}
\right.
$$

Using the fact that $\vp _n(x,0)>\alpha$ and so $\frac{\pa v_n}{\pa\nu}(x,0)=0$ for $x>0$, we deduce that for all $n$, we have
$$\vp _n (x,y) \geq w(x,y)=\vphi_{c_n}(x,y) - \eta(R_n) \qquad \mbox{ in $Q_{R_n}^+$}.
$$
and letting $n\to\infty$, we get
\begin{equation}\label{eq:ineq4} 
\vp _0 (x,y) \geq \vphi_{c_0}(x,y) \qquad \mbox{ in $\RR^2_+$}.
\end{equation}
This  yields in particular
$$ \lim_{x\rightarrow +\infty } \vp _0 (x,0)=\gamma^+=1,$$
(which proves the first part of Lemma \ref{lem:limitg}), and
$$ 1-\vp _0 (x,0) \leq 1- \vphi_{c_0} (x,0) =   \frac{2\sqrt c }{\sqrt{\pi}}\int_{\sqrt{x}}^\infty e^{-c_0 s^2}\,ds\quad
\mbox{ for $x>0$}$$
(which gives \eqref{eq:asymp}).

\vspace{15pt}

Next, we recall that (\ref{eq:limf}) implies that either $\gamma^-=\alpha$ or $\gamma^-=0$.
But if $\gamma^-=\alpha$, then $\vp _0(x,0) \geq \alpha$ on $\RR$ and so $\vp _0$ solves
$$
\left \{
\begin{array}{ll}
\Delta \vp _0+c_0 \pa_x \vp _0=0,& \mbox{ in $\RR^2_+ $}\\[5pt]
\dd \frac{\partial \vp _0}{\partial \nu} = 0 ,& \mbox{ on $\pa\RR^2_+$} 
\end{array}
\right.
$$
and is bounded in $\RR^2_+$. 
We deduce that $\vp _0$ is constant in $\RR^2_+$ (note that we could also have used Lemma \ref{lem:intc} to show this), which contradicts the fact that
$$\lim_{x\to-\infty } \vp _0 (x,0)=\alpha \neq 1 =\lim_{x\to+\infty } \vp _0 (x,0).$$
So we must have $\gamma^-=0$.
\vspace{15pt}

Finally, it only remains to show that  $\lim_{y\to +\infty} v_0(x,y)=1$.
The monotonicity of $v_0$ with respect to $y$ implies that there exists a function $\psi(x)$ such that
$\lim_{y\to +\infty} v_0(x,y)=\psi(x)$ for all $x \in \RR$.
It is readily seen that $\psi$ is bounded (between $0$ and $1$) and must solve
$ \Delta \psi -c_0\pa_x \psi =0$ in $\RR^2$. Liouville theorem thus gives $\psi(x)=\psi_0\in[0,1]$ for all $x$.
Finally, since $v_0$ is increasing with respect to $y$, we must have $\psi(x) \geq v_0(x,0)$ for all $x$ and so
$$ \psi_0 \geq \sup_{x\in\RR} v_0(x,0).$$
We deduce $\psi_0=1$
 which completes the proof of Lemma \ref{lem:limitg} and that of Theorem~\ref{thm:1}.
\end{proof}

\section{Asymptotic behavior of $v$}\label{sec:last}
In this section, we derive the  asymptotic formula (\ref{eq:asymp_exact}).
Using the fact that $\pa_y v_0(x,0)=0$ for all $x>0$, we can reflect  $v_0$ with respect to the $x$ axis and define 
$$
\tilde v (x,y) = 
\left\{\begin{array}{ll}
v_0(x,y) & \mbox{ if } y>0\\
v_0(x,-y) & \mbox{ if } y<0.
\end{array}
\right.
$$
This function  solves
$$
\Delta \tilde v +c\, \pa_x \tilde v =0,\qquad  \mbox{ in $\RR^2 \setminus \Gamma^0_-$ }
$$
where $\Gamma^0_-$ is the negative $x$-axis:
$$ \Gamma^0_-=\{(x,0)\in \RR^2\, ;\, x<0\}
$$

We now introduce the function
$$ w(x,y) = e^{\frac{c}{2} x}(1-\tilde v(x,y)).$$
Thanks to (\ref{eq:ineq4}), it is readily seen that $w$ is bounded in $\RR^2$, 
and a straighforward computation yields the following Helmholtz's equation:
$$
-\Delta w +\frac{c^2}{4} w =0,\qquad  \mbox{ in $\RR^2\setminus \Gamma^0_-$ .}
$$
Finally, we note that along $\Gamma^0_-$, $w$ is continuous, but the normal derivative satisfies
$$ w^\pm_y (x,0)= -\pm e^{\frac{c}{2}x} f(v_0(x,0))$$
where $w^\pm_y(x,0)$ denotes $\lim_{y\rightarrow 0^\pm} w_y(x,y)$.

Next, we recall that the fundamental solution of Helmholtz's equation, solution of
$$ -\Delta \phi + \frac{c^2}{4}  \phi = \delta$$
is given by
$$ \phi (r) = \frac{1}{2\pi} K_0\left(\frac{c}{2} r\right)$$
where $K_0(s)$ denotes the modified Bessel function of the second kind.
We also recall the following asymptotic behavior of $K_0$:
$$ K_0(s) \sim \sqrt{\frac{\pi}{2s}} e^{-s}+ \mathcal O\left(\frac{e^{-s}}{s^{3/2}}\right) \mbox{ as $s\to\infty$.}$$

Using these results, we can write for all $x>0$
\begin{eqnarray*}
w(x,0) &  = &  \int_{\pa ( \RR^2\setminus \Gamma^0_-) } \phi(|x-x'|) w_\nu (x')\, dx' \\
&  = &  -\int_{ \Gamma^0_- } \phi(|x-x'|) w^+_y (x')\, dx' \\
&   &  + \int_{ \Gamma^0_- } \phi(|x-x'|) w^-_y (x')\, dx' \\
& = &  2 \int_{-\infty}^0  \phi(|x-x'|) e^{\frac{c}{2}x'} f(v_0(x',0))\, dx'\\
& = &  2 \int_0^\infty  \phi(|x+x'|) e^{-\frac{c}{2}x'} f(v_0(-x',0))\, dx'.
\end{eqnarray*}
For large $x$, we deduce:
\begin{eqnarray*}
w(x,0) & = & \frac{1}{\sqrt{\pi c}} \int_0^\infty  \frac{1}{\sqrt{x+x'}} e^{-\frac{c}{2}(x'+x)} e^{-\frac{c}{2}x'} f(v_0(-x',0))\, dx'\\
&  & + \mathcal O\left( \int_0^\infty  \frac{1}{(x+x')^{3/2}}  e^{-\frac{c}{2}(x'+x)} e^{-\frac{c}{2}x'} f(v_0(-x',0))\, dx'\right) \\
& =  & \frac{1}{\sqrt{\pi c}} e^{-\frac{c}{2}x} \int_0^\infty  \frac{1}{\sqrt{x+x'}} e^{-c x'}  f(v_0(-x',0))\, dx'\\
&  & + \mathcal O\left(  \frac{e^{-\frac{c}{2}x}}{x^{3/2}} \right).
\end{eqnarray*}

Finally, we check that
\begin{eqnarray*} 
&& \int_0^\infty  \frac{1}{\sqrt{x+x'}} e^{-c x'}  f(v_0(-x',0))\, dx'  \\
&& \qquad = \frac{1}{\sqrt{x}} \int_0^\infty e^{-c x'}  f(v_0(-x',0))\, dx' + \mathcal O \left( \frac{1}{x^{3/2}}  \int_0^\infty x'  e^{-c x'}  f(v_0(-x',0))\, dx'  \right)
\end{eqnarray*}
and we deduce (recall that $f$ is bounded)
$$
w(x,0)=   \mu_0 \frac{e^{-\frac{c}{2}x}  }{\sqrt x} + \mathcal O\left(  \frac{e^{-\frac{c}{2}x}}{x^{3/2}} \right) \quad \mbox{ as $x\rightarrow \infty$}
$$
with
$$\mu_0=  \frac{1 }{\sqrt{\pi c}}\int_0^\infty e^{-c x'}  f(v_0(-x',0))\, dx' .$$
This implies 
$$
1-v(x,0) =\mu_0  \frac{e^{-c x}  }{\sqrt x} + \mathcal O\left(  \frac{e^{-cx}}{x^{3/2}} \right) \quad \mbox{ as $x\rightarrow \infty$}
$$
and completes the proof of Proposition \ref{prop:as}.

\bibliographystyle{alpha} 
\bibliography{biblio} 

\end{document}